\documentclass[11pt,twoside,reqno]{amsart}
\usepackage[all]{xy}
        \usepackage {amssymb,latexsym,amsthm,amsmath,mathtools,dsfont,mathrsfs}
        \usepackage{enumitem,color}

\usepackage{mathtools}

\usepackage{biblatex} 
\usepackage{hyperref}
\long\def\symbolfootnote[#1]#2{\begingroup%
\def\thefootnote{\fnsymbol{footnote}}\footnote[#1]{#2}\endgroup}

\makeatletter
\def\imod#1{\allowbreak\mkern10mu({\operator@font mod}\,\,#1)}
\makeatother
\makeatletter
\renewcommand*\env@matrix[1][*\c@MaxMatrixCols c]{%
  \hskip -\arraycolsep
  \let\@ifnextchar\new@ifnextchar
  \array{#1}}
\makeatother
\usepackage[capitalize,nameinlink]{cleveref} 
\usepackage{comment} 
\usepackage{xcolor} 
\hypersetup{
    colorlinks,
    linkcolor={red!80!black},
    citecolor={green!80!black},
    urlcolor={blue!80!black}
}
\newtheorem{theorem}{Theorem}[section]
\newtheorem{lemma}[theorem]{Lemma}

\newtheorem{proposition}[theorem]{Proposition}
\newtheorem*{theorem*}{Theorem}
\theoremstyle{definition}

\newtheorem{remark}[theorem]{Remark}

\numberwithin{equation}{section}
\newcommand{\ignore}[1]{}

\newcommand{\mynote}[1]{}
\newcommand{\Z}{\mathbb{Z}}
\newcommand{\hol}{\textup{Hol}}
\newcommand{\aut}{\textup{Aut}}
\newcommand{\G}{{G}}
\newcommand{\C}{{C}}
\newcommand{\N}{{N}}
\newcommand{\fa}{\mathfrak{f}}
\newcommand{\ga}{\mathfrak{g}}
\newcommand{\ha}{\mathrm{Hom}}
\newcommand{\ma}{\mathrm{Map}}
\newcommand{\pe}{\mathrm{Perm}}
\title[Hopf Galois structures and skew braces for groups of size $p^nq$]{Hopf Galois structures, skew braces for groups of size $p^nq$: The cyclic Sylow subgroup case}
\author[Arvind N.]{Namrata Arvind}
\email{namchey@gmail.com}
\address{The Institute of Mathematical Sciences, 4th Cross St, CIT Campus, Tharamani, Chennai, Tamil Nadu 600113, India}
\author[Panja S.]{Saikat Panja}
\email{panjasaikat300@gmail.com}
\address{Harish-Chandra Research Institute- Main Building, Chhatnag Road, Jhusi, Uttar Pradesh 211019, India}
\thanks{The first named author is partially supported by the IMSc postdoctoral fellowship and the second author has been partially supported by HRI postdoctoral fellowship.}
\date{\today}
\begin{document}
\setcounter{section}{0}
\begin{abstract}
Let $n\geq 1$ be an integer, $p$, $q$ be distinct odd primes. Let ${G}$, $N$ be two
groups of order $p^nq$ with their Sylow-$p$-subgroups being cyclic. 
We enumerate the Hopf-Galois structures on a Galois ${G}$-extension, with type $N$. 
 This also computes the number of skew braces with additive group isomorphic to
 $G$ and multiplicative group isomorphic to $N$.
 Further when $q<p$, we give a complete classification of the Hopf-Galois structures on Galois-$G$-extensions.
\end{abstract}
\maketitle
\section{Introduction}\label{sec:introduction}
The study of Hopf-Galois structures comes under the realm of group theory and 
number theory. 
These structures were first studied by S. Chase and M. Sweedler in $1969$, in their work \cite{ChSw69}. Subsequently in \cite{GrPa87}, C. Greither 
and B. Pareigis defined a Hopf-Galois structure for separable extensions. 
In recent times, algebraic objects called Skew braces  were introduced in
the PhD thesis of D. Bachiller. They have been studied by 
various mathematicians like W. Rump, D. Bachiller, F. Cedo in \cite{CeJeOk14}, \cite{BaCeJe16} \emph{etc.}. Skew braces are 
known to give set-theoretic solutions to the Yang-Baxter equations. 
Subsequently, A. Smoktunowicz and L. Vendramin noticed a connection between the study of skew braces and that of Hopf-Galois structures in their work
\cite{SmVe18}. For more details, and background on this topic, and the interplay between skew braces and Hopf-Galois structures, one can refer to the book \cite{Ch00book} of L. N. Childs and the Ph.D. thesis of K. N. Zenouz \cite{Ze18}.

A Hopf-Galois structure on a finite field extension is defined in the following way.
Assume $K/F$ is a finite Galois field extension. An $F$-Hopf algebra $\mathcal{H}$, with an action on $K$ such that $K$ is an $H$-module algebra
and the action makes $K$ into an $\mathcal{H}$-Galois extension, will be called a \textit{Hopf-Galois structure}
on $K/F$.

A \emph{left skew brace} is a triple $(\Gamma,+,\times)$ where $(\Gamma,+),(\Gamma,\times)$ are groups and satisfy
$a\times(b+c)=(a\times b)+a^{-1}+(a\times c),$ for all $a,b,c\in\Gamma$.

Given a group $G$, the $Holomorph$ of $G$ is defined as $G\rtimes \aut(G)$, via the identity map. It is denoted by $\hol(G)$.
Let $G$ and $N$ be two finite groups of the same order. By $e(G,N)$ we mean the number of Hopf-Galois structures on a finite Galois field extension $L/K$ with Galois group isomorphic to $G$, and the type isomorphic to $N$.  In \cite{GrPa87}, the authors gave a bijection between Hopf-Galois structures on a finite Galois extension with Galois group $G$ and regular subgroups in $\pe(G)$, which are normalised by $\lambda(G)$.  Further in \cite{By96}, N. Byott showed that 
\begin{equation}\label{Byotts-translation}
    e(G,N)= \dfrac{|\aut(G)|}{|\aut(N)|}\cdot e'(G,N),
\end{equation}
where $e'(G,N)$ is the number of regular subgroups of $\hol(N)$ isomorphic to $G$.
Here a subgroup $\Gamma$ of $\hol(N)$ is called regular if it has exactly one element $(e_G,\zeta)\in \Gamma$ with $\zeta=I$, the identity 
automorphism. We will use this condition to check regular embeddings of the concerned groups in the article. 
It turns out that $e'(G, N)$ also gives the number of Skew-Braces with the additive group isomorphic 
to $N$ and the multiplicative group isomorphic to $G$.
The number $e(G,N)$ has been computed for several groups. For example, N. Byott determined $e(G,N)$ when $G$ is isomorphic to a cyclic group \cite{By13}; C. Tsang 
determined $e(G,N)$ when $G$ is a quasisimple group \cite{Ts21a};  N. K. Zenouz
consider the groups of order $p^3$ \cite{Ze18} to determine $e(G,N)$ ; T. Kohl
determined $e(G,N)$ when $G$ is a dihedral group \cite{Ko20}.

Previously in \cite{ArPa22}, the authors computed $e(G,N)$ whenever $G$ and $N$ are isomorphic to $\Z_n\rtimes \Z_2$, where $n$ is odd and its radical is a Burnside number. 
 Groups of order $p^2q$ with cyclic Sylow subgroups have been considered in \cite{CaCaDe20}. 
 We can show that any group of order $p^nq$ with cyclic Sylow subgroups, when $p$ and $q$ are distinct primes, is a semidirect product of two cyclic groups (see \cref{sec:prelim}).
 In this article, we compute $e(G,N)$ (and $e'(G,N)$), 
 whenever $G$ and $N$ are groups of order $p^n q$ with cyclic Sylow-$p$ subgroup, 
 where $p$ and $q$ are distinct odd primes. We do this by looking at the number 
 of regular subgroups of $\hol(N)$ which are isomorphic to $G$. Finally whenever 
 $q < p$ we give a necessary and sufficient condition on when the pair $(G,N)$ is realizable.

We now fix some notations.
For a ring $R$, we will use $R^{\times}$ to denote the
set of multiplicative units of $R$. For a group $G$, the identity element will be sometimes denoted by $e_G$ and mostly by $1$, when the context is clear. 
The automorphism group of a group $G$ will be denoted by $\aut(G)$, 
and the holomorph $G\rtimes_{\mathrm{id}}\aut(G)$ will be denoted by $\hol(G)$.
The binomial coefficients will be denoted by ${l \choose m}$. The Euler totient function 
will be denoted by $\varphi$. We will use $\Z_m$ to denote the cyclic group of order $m.$ We will use $\Z_m$ as a group as well as a ring, which will be clear from the context.
Now, we state the two main results of this article. 
To state the second result we use notations from \cref{sec:prelim}.
\begin{theorem}\label{theorem-p->-q}
    Let $p>q$ be odd primes and $q|p-1$. Let $G$ denote the nonabelian group of the form $\Z_{p^n}\rtimes\Z_q$ and $C$ denote the cyclic group of order $p^nq$. Then the following are true:
    \begin{enumerate}
        \item $e'(G,G)=e(G,G)=2+2p^n(q-2)$,
        \item $e'(G,C)=q-1$, and $e(G,C)=p^n$,
        \item $e'(C,G)=p^{2n-1}$, and $e(C,G)=2p^{n-1}(q-1)$.
    \end{enumerate}
\end{theorem}
\begin{theorem}\label{theorem-p-<-q}
    Let $p<q$ be odd primes and $p^a||q-1$. For $1\leq b\leq\mathrm{min}\{n,a\}$, let $G_b$ denote the unique nonabelian 
    group of the form $\Z_{p^n}\rtimes\Z_{q}$ determined by $b$, 
    and $C$ denote the cyclic group of $p^nq$. Then the following results hold;
    \begin{enumerate}
        \item $e'(G_b,G_b)=e(G_b,G_b)=2\left(p^{n-b}+q\left(\varphi(p^n)-p^{n-b}\right)\right)$,
        \item $e'(G_{b_1},G_{b_2})=2qp^{n+b_1-b_2-1}(p-1)$, and $e(G_{b_1},G_{b_2})=2qp^{n-1}(p-1)$ for $b_1\neq b_2$,
        \item $e'(C,G_b)=2p^{n-b}q$, and $e(C,G_b)=2(p-1)p^{n-1}$,
        \item $e'(G_b,C)=p^{n+b-2}(p-1)$, and $ e(G-b,C)=p^{n-1}b$.
    \end{enumerate}
\end{theorem}
The rest of the article is organised as follows.
In \cref{sec:prelim}, we give a detailed description of  the groups under consideration and determine their automorphism groups.
Next, in \cref{sec:main-results-1} and \cref{sec:main-results-2} we will prove \cref{theorem-p->-q} and \cref{theorem-p-<-q} respectively.
Lastly, in \cref{sec:conclusion} we discuss the realizability problem and solve them for some of the groups mentioned in  this article.
\section{Preliminaries}\label{sec:prelim}
\subsection{The groups under consideration}
In this subsection we will describe the groups under consideration and fix
some notations. 
Let $p$ and $q$ be distinct odd primes.
We look at groups of order $p^nq$ whose Sylow-$p$-subgroups are cyclic.
These come under two families, depending on whether $p>q$ or $p<q$.

In case $p>q$, the
groups are isomorphic to $\Z_{p^n}\rtimes \Z_{q}$, since $\Z_{p^n}$ is normal. Indeed all these
groups $G$ fits into the short exact sequence $1\longrightarrow \Z_{p^n}\longrightarrow G\longrightarrow \Z_q\longrightarrow 1$. Thus by the well known Schur–Zassenhaus theorem $G$ is isomorphic to $\Z_{p^n}\rtimes \Z_{q}$. Since $\aut(\Z_{p^n})$ is cyclic, the semidirect product is either trivial (in this case the group is cyclic)
or uniquely nontrivial. Let $G \cong \Z_{p^n} \rtimes \Z_q$.
If $q\nmid p-1$ then $G$ is cyclic.
In case $q|p-1$, let $\phi : \Z_q \rightarrow \aut(\Z_{p^n})$ 
be the homomorphism defined as $\phi(1)=k$. 
Here $k$ is an element of  $\aut(\Z_{p^n})$ of order $q$. 
Hereafter, we denote $\Z_{p^n} \rtimes_{\phi} \Z_q$ by $\Z_{p^n} \rtimes_k \Z_q$. 
Let $$\{x,y| x^{p^n}=y^q=1 , yxy^{-1} = x^k\}$$ be a presentation of $\Z_{p^n} \rtimes_k \Z_q$. 
Note that since $e(G,G)$ is already known whenever $G$ is cyclic, we will assume $q|p-1$ for our calculations.

Now if $p < q $ we need to use a result of W. Burnside from \cite{Bu05}, which states that 
for a finite group $G$, all the Sylow subgroups are cyclic if and only if $G$ is a
semidirect product of two cyclic groups of coprime order. Applying this to our situation, we get that $G$ is either 
a cyclic group or a non-trivial semidirect product of the form $\Z_q\rtimes\Z_{p^n}$. Next, we elaborate on different possible semidirect products of the form $\Z_q\rtimes\Z_{p^n}$
. Once 
again in this case we assume that $p|q-1$. Let $p^a||q-1$ and for $b\leq \mathrm{min}\{n,a\}$ fix $\psi_{b}:\Z_{p^n}\longrightarrow\aut(\Z_q)$ to be a homomorphism, such that
$|\mathrm{Im}~\psi_{b}|=p^b$. Take $\G_{b}=\Z_q\rtimes_{\psi_{b}}\Z_{p^n}$. 
The group $\G_{b}$ is unique up to isomorphism. The presentation of this group can be taken to be
$$\langle x,y|x^{p^n}=1,y^q=1,xyx^{-1}=y^k\rangle,$$
where $k$ is an element of order $p^b$ in $\aut(\Z_q)=\Z_q^\times$. From now on we denote $\Z_q\rtimes_{\psi_{b}}\Z_{p^n}$ by $\Z_q\rtimes_{k}\Z_{p^n}$.

\subsection{The basic lemmas} In this subsection we note down the basic group-theoretic results,
which will be used throughout the article.
\begin{lemma}\label{lem-power-of-p}
    Let $p$ be a positive odd integer. Take $a=bp^c$ where $p\nmid b$. Then we have that $(1+p)^{a}\equiv 1+dp^{c+1}\pmod{p^{c+2}}$ for some $p\nmid d$, for all integer $c\geq 0$.
\end{lemma}
\begin{proof}
    We prove it by induction on $c$. If $c=0$, then $(1+p)^a=1+ap+a'p^2$, for some
    $a'\in\Z$. Hence $(1+p)^{a}\equiv 1+ap\pmod{p^2}$ with $d=a$. Next, assume it 
    to be true for all $l\leq c$ and in particular for $l=c$. Hence 
    $(1+p)^{bp^c}=1+dp^{c+1}+d'p^{c+2}$ for some $d'\in\Z$. Then we have 
    \begin{align*}
        (1+p)^{bp^{c+1}}=\left(1+dp^{c+1}+d'p^{c+2}\right)^p=(1+d''p^{c+1})^p,
    \end{align*}
    for some $d''\in Z$ and $(d'',p)=1$. Hence it follows that $(1+p)^{bp^{c+1}}\equiv 1+ d'' p^{c+2}\pmod{p^{c+3}}$, which also 
    finishes the induction, and hence the proof.
\end{proof}
\begin{lemma}\label{lem:holaut}
Let $G$ be the non-abelian group isomorphic to $\Z_{p^n}\rtimes \Z_{q}$. We have $\aut(G) \cong \hol(\Z_{p^n})$.
\end{lemma}
\begin{proof} 
We first embed $G$ as a normal subgroup of  $\hol(\Z_{p^n})$. 
Take the homomorphism $\psi$ defined as
\begin{align*}
\psi(x)&=\begin{pmatrix}1&k\\0&1\end{pmatrix},~
\psi(y)=\begin{pmatrix}k&1\\0&1\end{pmatrix}.
\end{align*}
 This embedding can be shown to be injective. 
 Now consider the following map 
 \begin{align*}
 \Phi: \hol(\Z_{p^n}) \longrightarrow \aut(G) 
 \text{ defined as }\Phi(z)(w)=zwz^{-1}
 \end{align*}
  for all $z\in \hol(\Z_{p^n})$ and $y\in G$ is an injective group homomorphism, since $\ker\Phi$ consists only of the identity matrix. From \cite[Theorem B]{Walls1986} we have $|\aut(G)|= |\hol(\Z_{p^n})|$. Thus $\Phi$ is an isomorphism.
\end{proof}
\begin{lemma}\label{prop:unit}
Let $p,q$ be primes such that $(p,q)=1$ and $q|p-1$.
Let $k$ be a multiplicative unit in $\Z_{p^n}$, of multiplicative order $q$.
Then $k-1$ is a multiplicative unit in $\Z_{p^n}$. 
\end{lemma}
\begin{proof} 
Suppose $k-1$ is not a unit in $\Z_{p^n}$, then $k-1 = mp$ 
for some $m\in \Z_{p^n}$.
Since $(k)^q\equiv 1\pmod{p^n}$, we get
\begin{align*}
(mp+1)^q \equiv 1+ {q \choose 1}mp + {q \choose 2}(mp)^2+ \cdots +(mp)^q \equiv 1\pmod{p^n},
\end{align*}
which in turn implies that
\begin{align*}
mp\cdot\left(q+ {q \choose 2}mp + {q \choose 3}(mp)^2+ \cdots +(mp)^{q-1}\right) \equiv 0 \pmod{p^n}.
\end{align*}
We note that 
\begin{align*}
t=q+ {q \choose 2}mp + {q \choose 3}(mp)^2+ \cdots +(mp)^{q-1} 
\end{align*}
 is a unit since $q$ is a unit and $t-q$
 is a nilpotent element. Thus 
$mp \equiv 0 \pmod{p^n},$
which implies $k-1 \equiv 0 \pmod{p^n}$. This is a contradiction since $k$ is an element of order $q$.
\end{proof}
\begin{lemma}\label{prop:structureAutomorphism}
Let  $G_b \cong \Z_q\rtimes_{k}\Z_{p^n}$, where $k$ is an element of order $p^b$ in $\Z_{q}^{\times}$. Assume $p|q-1$, then for $b>0$, we have that $\aut(\G_{b})\cong\Z_{p^{n-b}}\times\hol(\Z_q)$.
\end{lemma}
\begin{proof}
The proof will be divided into two steps. First, we calculate the size of the automorphism group.
In the next step, we will determine the group's description in terms of generators and relations, from which the result will follow.

Let us take an automorphism $\Psi$ of $\G_b$. Since an automorphism is determined
 by its value on the generator, assume that $\Psi(x)=y^\alpha x^\gamma$ and $\Psi(y)=y^{\beta}x^{\delta}$, where $0\leq \alpha,\beta\leq q-1$ and $0\leq \gamma,\delta\leq p^n-1$. Note that we have
 $\Psi(y)^q=y^{\beta(1+k^\delta+k^{2\delta}+k^{(q-1)\delta})}x^{q\delta}$. Since 
 $\Psi(y)^q=1$, we must have $\delta=0$. \textcolor{blue}{Thus $\beta$ should be a unit in $\Z_q$}. Now 
 consider the equation $\Psi(x)\Psi(y)=\Psi(y)^k\Psi(x)$. This imposes the condition
 that
 $y^{\alpha+\beta k^{\gamma}}x^{\gamma}=y^{\beta k+\alpha}x^{\gamma}$. Hence we should have
 $\beta k^{\gamma}\equiv\beta k\pmod{q}$, whence $k^{\gamma-1}\equiv 1\pmod{q}$,
 as $\beta$ is a unit in $\Z_q$. Since $k$ is an element of order $p^b$, \textcolor{blue}{we get that $\gamma\in\{Rp^b+1:0\leq R<p^{n-b}\}$}.
Next considering the equation $\Psi(x)^{p^n}=1$, we have that 
 $y^{\alpha(1+k^{\gamma}+k^{2\gamma}+\ldots+k^{(p^{n}-1)\gamma})}x^{p^n\gamma}=1$.
 Since $x^{p^n\gamma}=1$, we have that 
 $\alpha(1+k^{\gamma}+k^{2\gamma}+\ldots+k^{(p^{n}-1)\gamma})=0\pmod{q}$. Regardless of the value of $k$, \textcolor{blue}{any $0\leq\alpha\leq q$} satisfies the last congruence. Hence the group is of order $p^{n-b}q(q-1)$.

Hereafter we denote $\Psi$ by $(\gamma,\beta,\alpha)$. Consider the following three elements of the group given by
\begin{align*}
    \Psi_1=\left((1+p)^{p^{b-1}},1,0\right), \Psi_2=\left(1,t,0\right), \Psi_3=(1,1,1),
\end{align*}
where $1\leq t\leq q-1$ satisfies $\Z_q^\times=\langle\overline{t}\rangle$. Since
$\overline{(1+p)}\in\Z_{p^n}^\times$ is of order $p^{n-1}$, we get that
$\Psi_1$ is an element of order $p^{n-b}$. Given that, $\overline{t}$ is an element of order
$q-1$, the element $\Psi_2$ is of order $q-1$. Lastly, $\Psi_3$ is an element of order $q$.
Note that \textcolor{blue}{$\Psi_1\Psi_2=\Psi_2\Psi_1$}, follows from an easy calculation. Now, $\Psi_1\Psi_3(x)=yx^{(1+p)^{p^{b-1}}}$. Further, we have,
\begin{align*}
    \Psi_3\Psi_1(x)=(yx)^{(1+p)^{p^{b-1}}}=y^{1+k+\ldots+k^{\left(1+p\right)^{p^{b-1}}}}x^{(1+p)^{p^{b-1}}}=y^{1+\left(\frac{k^{(1+p)^{p^{b-1}}-1}}{k-1}\right)}x^{k^{(1+p)^{p^{b-1}}-1}}.
\end{align*}
Since ${(1+p)^{p^{b-1}}-1}\equiv 1\pmod{p^b}$ and $\overline{k-1}$ is a unit in 
$\Z_q$, we conclude that $\Psi_1\Psi_3(x)=\Psi_3\Psi_1(x)$. Since $\Psi_1
\Psi_3(y)=\Psi_3\Psi_1(y)$, we conclude that \textcolor{blue}{$\Psi_1\Psi_3=\Psi_3\Psi_1$}.
We now take the subgroup generated by $\Psi_2$ and $\Psi_3$. In this group
$\langle\Psi_3\rangle$ is normal as $\Psi_2\Psi_3\Psi_2^{-1}=\Psi_3^t\in\langle\Psi_3\rangle$. Also $\langle\Psi_2\rangle\cap\langle\Psi_3\rangle$ contains
only identity. Hence $|\langle\Psi_2,\Psi_3\rangle|=q(q-1)$. Take the map
$T:\langle \Psi_2,\Psi_3\rangle\longrightarrow\hol(\Z_q)$, defined as
\begin{align*}
    T(\Psi_2)=\begin{pmatrix}
        t&0\\
        0&1
    \end{pmatrix}\text{ and }
    T(\Psi_3)=\begin{pmatrix}
        1&1\\
        0&1
    \end{pmatrix}.
\end{align*}
This determines a homomorphism since $T(\Psi_2)T(\Psi_3)T(\Psi_2)^{-1}=T(\Psi_3)^t$. For any
$\begin{pmatrix}
    u & v\\ 0 &1
\end{pmatrix}\in\hol(\Z_q)$, we have that $T(\Psi_2^{w_1}\Psi_3^{w_2})=\begin{pmatrix}
    u & v \\ 0 & 1
\end{pmatrix}$, where $w_1$ satisfies $t^{w_1}=u$ and $w_2=v/u$. Since the order 
of the groups are the same, we conclude that \textcolor{blue}{$\langle\Psi_2,\Psi_3
\rangle\cong \hol(\Z_q)$}. Now we will show that $\langle\Psi_1\rangle\cap\langle
\Psi_2,\Psi_3\rangle$ has only the identity element. Indeed, if $\Psi_1^d=
\Psi_2^e\Psi_3^f$ (for some $0\leq d< p^{n-b}$, $0\leq e<q-1$ and $0\leq f<q$), 
then $e=0$, comparing the evaluation of both the functions at $y$. 
Finally, if we consider $\Psi_1^d(x)=\Psi_3^f(x)$, we get that $x^{p'^d}=y^fx$ where $p'=(1+p)^{p^{b-1}}$. This forces us to have $f=0$, consequently $d=0$.
Thus \textcolor{blue}{$\langle \Psi_1,\Psi_2,\Psi_3\rangle\cong\langle\Psi_1\rangle\times\langle\Psi_2,\Psi_3\rangle$} and \textcolor{blue}{is of order $p^{n-b}q(q-1)$}. Hence we have proved that $\aut(\G_b)\cong\Z_{p^{n-b}}\times\hol(\Z_q)$.
\end{proof}
We denote the elements of $\aut(G_b)$ by $\left(\gamma,\begin{pmatrix}
    \beta & \alpha \\ 0 & 1
\end{pmatrix}\right)\in \Z_{p^{n}}^{\times}\times \hol(\Z_q)$, such that 
$\gamma^{p^{n-b}}=1$.
\begin{remark}
    We note down the action of the automorphism group of $\G_b$ on the group $\G_b$, by means of generators. This will be necessary for counting the Hopf-Galois structures concerning $\G_b$'s. For $b>0$, the action is as follows.
    \begin{align*}
        \left(\gamma,\begin{pmatrix}
            \beta & \alpha\\ 
            0  & 1
        \end{pmatrix}\right)\cdot x= y^{\alpha}x^{\gamma}\text{ and, }
        \left(\gamma,\begin{pmatrix}
            \beta & \alpha\\ 
            0  & 1
        \end{pmatrix}\right)\cdot y= y^{\beta}.      
    \end{align*}
\end{remark}
\begin{remark}
    For $b=0$, the group $\G_b\cong\Z_{p^n}\times \Z_q$. Since $(p,q)=1$ and both 
    are abelian groups, it follows from \cite[Theorem 3.2]{BiCuMc06} that 
    $\aut(\G_b)\cong\Z_{p^{n-1}(p-1)}\times\Z_{q-1}$ in this case. The action is 
    defined to be component-wise.
\end{remark}

\section{The case $ p>q$}\label{sec:main-results-1}
This section is devoted to the proof of \cref{theorem-p->-q}. As discussed in \cref{sec:prelim}, upto isomorphism 
there are precisely two groups of order $p^nq$ whenever their Sylow subgroups are cyclic. 
Counting the number of skew braces with multiplicative group $G$ and additive group $N$ is equivalent to (up to multiplication by a constant; see \cite[Proof of Proposition 3.2]{ArPa22})
counting the number of regular embedding of $G$ in $\hol(N)$. Then using \cref{Byotts-translation}, we are able to conclude about the number of Hopf-Galois structures on $G$-extensions of type $N$.
We will use the regularity criterion given in \cref{sec:introduction}.
This section will be divided into three subsections, depending on the isomorphism types of $G$ and $N$.
From \cref{lem:holaut}, we have that  $\aut(\Z_{p^n} \rtimes_k \Z_q) \cong \hol(\Z_{p^n})$, where the action is given by,
\begin{align*}
\begin{pmatrix}
\beta & \alpha\\
0 & 1 
\end{pmatrix}
\cdot x^iy^j = x^{\beta i +\alpha k^{-1}-\alpha k^{j-1}}y^j.
\end{align*}

\subsection{Embedding of $\Z_{p^n} \rtimes_k \Z_q$ into $\hol(\Z_{p^n} \rtimes_k \Z_q)$}\label{subsec:eGG}
Let $\Phi : \Z_{p^n} \rtimes_k \Z_q \longrightarrow \hol(\Z_{p^n} \rtimes_k \Z_q)$ be a regular embedding. Let 
\begin{align*}
\Phi(x)= \left(x^{i_1}y^{j_1},
\begin{pmatrix}
\beta_1 & \alpha_1\\
0 & 1 \end{pmatrix}\right),
\Phi(y)= \left(x^{i_2}y^{j_2},
\begin{pmatrix}
\beta_2 & \alpha_2\\
0 & 1 \end{pmatrix}\right).
\end{align*}
From $(\Phi(x))^{p^n}=1$ we get
\begin{align}
   j_1\equiv 0 \pmod{q},\label{e1} 
\end{align}
since $p^nj_1\equiv 0\pmod{q}$ and $(p,q)=1$,
\begin{align}
    \beta_1^{p^n}\equiv 1&\pmod{p^n},\label{e2}\\
   i_1(1+\beta_1+\beta_1^2+\ldots +\beta_1^{p^n-1})\equiv 0 &\pmod{p^n},\label{e3}\\
   \alpha_1(1+\beta_1+\beta_1^2+\ldots +\beta_1^{p^n-1})\equiv 0 &\pmod{p^n}.\label{e4}
\end{align}
Similarly from $\Phi(yxy^{-1})= \Phi(x^k)$ we get 
\begin{align}
  \beta_1^{k-1} &\equiv 1 \pmod{p^n},\label{f1}
\end{align}
which implies $\beta_1=1$ from \cref{e2}, \cref{f1} and using \cref{prop:unit}; furthermore, 
\begin{align}
 \beta_2\alpha_1 + \alpha_2 &\equiv \beta_1^k\alpha_2 + \alpha_1 \pmod{p^n},\label{f2}\\
    ki_1\left(k^{j_2-1}\beta_2-1\right)&\equiv \alpha_1\left(1-k^{j_2}\right)\pmod{p^n}.\label{f3}
\end{align}
Further taking $\beta_1 = 1$ in \cref{f2} and \cref{f3} we get that,
\begin{align}
\alpha_1 \cdot(k-\beta_2) \equiv  0 \pmod{p^n},\label{f4}\\
ki_1\cdot(k^{j_2-1}\beta_2-1) \equiv \alpha_1\cdot (1-k^{j_2}) \pmod{p^n}.\label{f5}
\end{align}
We note that 
in general,
\begin{align*}
    \Phi(y)^k=\left(x^{\ell_k} y^{kj_2},\begin{pmatrix}
        \beta_2^k& \alpha_2(1+\beta_2+\beta_2^2+\cdots+\beta_2^{k-1})\\ & 1
    \end{pmatrix}\right),
\end{align*}
where 
\begin{align}
\ell_k = i_2\left(\sum\limits_{t=0}^{k-1}\left(\beta_2k^{j_2}\right)^t\right)+\left(\alpha_2k^{j_2-1}-\alpha_2k^{2j_2-1}\right)
\left(1+\sum\limits_{u=1}^{k-2}
\left(\sum\limits_{v=0}^u\beta_2^v\right)k^{uj_2}\right)\label{g}.
\end{align}

 Using $\Phi(y)^q =1 $ we get 
 \begin{align}
 \beta_2^q &\equiv 1 \pmod{p^n},\label{g1}\\
 \alpha_2(1+\beta_2+\beta_2^2+\ldots +\beta_2^{q-1})&\equiv 0 \pmod{p^n}&&j_2\neq 0,\label{g2}\\
 \ell_q &\equiv 0 \pmod{p^n}.\label{g3}
  \end{align}
From \cref{g1} we get $\beta_2 = k^a$, for some $0\leq a\leq q-1$, since $\Z_{p^n}^*$ has a unique subgroup of order $q$ and is generated by $k$.
First let us show that, in any regular embedding $j_2\neq 0$.
If possible let $j_2=0$. Then we get that $\beta_2=k$. This forces that for any $0\leq \omega_1\leq p^n-1$ and $0\leq \omega_2\leq q-1$
\begin{align}
    \Phi(x)^{\omega_1}\Phi(y)^{\omega_2}
    =\left(x^{\omega_1i_1+i_2\left(1+k+\cdots+k^{\omega_2-1}\right)},\begin{pmatrix}
        k^{\omega_2}& \star\\
        0 & 1
    \end{pmatrix}\right).\label{disc:regularity}
\end{align}
Since $i_1$ is a unit, making a suitable choice of $\omega_1$ and $\omega_2$ we get that this embedding will not be regular.
Indeed note that $1-k$ and $1-k^{\omega_2}$ are both units and so is $1+k+\cdots+k^{\omega_2-1}$. We now divide the possibilities of $a$ into $3$ cases.
\subsubsection{\textbf{Case I: $a=0$}}
Using \cref{f3} and \cref{f4}, we conclude that $\alpha_1 \equiv 0\pmod{p^n}$, $j_2 \equiv 1 \pmod{q}$ and, $\alpha_2\equiv 0\pmod{p^n}$.
Since $i_1$ is a unit in $\Z_{p^n}$ and $i_2\in\Z_{p^n}$ can take any value, the total number of embeddings in this case
is given by $p^n\varphi(p^n)$.
Moreover, all of these embeddings are regular. {We remark that all the
above embedding corresponds to the canonical Hopf-Galois structure.}

\subsubsection{\textbf{Case II: $a=1$}} Note that using \cref{f5} we get that $ki_1\equiv -\alpha_1\pmod{p^n}$. We deal with this in two subcases depending on the value of $j_2$. 
First, we consider the case $j_2$ being equal to $q-1$. 
In this case using $\ell_q=0$, we get that $i_2$ gets determined by
the value of $\alpha_2$ since $\left(\sum\limits_{t=0}^{k-1}\left(\beta_2k^{j_2}\right)^t\right)=q$ is a unit in $\Z_{p^n}$. 
Hence the number of embedding in this subcase is given by $p^n\varphi(p^n)$.

For the other case, since the element $k^{j_2}(1-k^a) $ is a unit and $j_2+a\neq 0 \pmod{q}$ we get
\begin{align*}
&1+\sum\limits_{s=1}^{q-2}\left(\sum\limits_{t=0}^{s}k^{ta}\right)k^{sj_2}
    =\dfrac{1}{k^{j_2}(1-k^a)}\left\{\sum\limits_{t=1}^{q-1}\left(1-k^{ta}\right)k^{tj_2}\right\}
    =\dfrac{1}{k^{j_2}(1-k^a)}\cdot (1-1)
    =0,
\end{align*} 
Thus $\Phi(y)^q=1$ 
does not impose any conditions on $i_2$ and $\alpha_2$. Hence, in this subcase, the total number of possibilities is $p^{2n}\varphi(p^n)(q-2)$. 
Since $j_2\neq 0$, we conclude that all the embeddings are regular.
\subsubsection{\textbf{Case III: $a\geq 2$}} 
This conditions together with \cref{f4} and \cref{f5}, imply that $\alpha_1=0$ and $j_2\equiv a-1\pmod{q}$. Since $a+j_2\not\equiv 0\pmod{q}$, 
a mutatis mutandis of Case II gives that $i_2$ and $\alpha_2$ can be chosen independently, whence 
each of them has $p^n$ possibilities. Thus, in this case, the total number of possibilities
is given by $p^{2n}\varphi(p^n)(q-2)$.
Similar to the previous case, all the embeddings are regular.  

Summarizing the above cases we get the following result.
\begin{lemma}\label{prop:regular-embeddings}
    The total number of regular embeddings of $\Z_{p^n}\rtimes\Z_{q}$ inside $\hol(\Z_{p^n}\rtimes\Z_{q})$ is given by $2p^n\varphi(p^n)+2p^{2n}\varphi(p^n)(q-2)$.
\end{lemma}
\begin{proposition}\label{thm:total-isomorphic}
    Let $\G$ be a non-abelian groups of the form $\Z_{p^n}\rtimes\Z_q$, where $p$ and $q$ are primes satisfying $q|p-1$. Then $e(\G,\G)$ is given by 
        $2+2p^n(q-2)$.
\end{proposition}
\begin{proof}
 From \cref{prop:regular-embeddings} we get the total number of regular embeddings. Dividing this number by the Automorphism of $G$ will give us the total number of Hopf-Galois structures.  
\end{proof}
\subsection{Embedding of $G=\Z_{p^n}\rtimes \Z_q$ in the $\hol(\Z_{p^n}\times \Z_q)$}\label{subsec:eGC}
Next we consider the case of regular embedding of $G=\Z_{p^n}\rtimes \Z_q$ in the $\hol(\Z_{p^n}\times \Z_q)$.
Let us fix the presentation of $C=\Z_{p^n}\times \Z_q$ to be 
$\langle r,s|r^{p^n}=s^q=1,rs=sr\rangle.$ Then it can be shown that
$\hol(C)\equiv\hol(\Z_{p^n})\times\hol(\Z_q)$.
We take a typical element of 
$\hol(C)$ to be $\left(\begin{pmatrix}
    b&a\\0&1
\end{pmatrix},\begin{pmatrix}
    d&c\\0&1
\end{pmatrix}\right)$, where $a$, $c$ are elements of $\Z_{p^n}$, $\Z_q$ respectively and
$b$, $d$ are elements of $\Z_{p^n}^\times$, $\Z_q^\times$ respectively.
Starting with an embedding $\Phi$ of $G$ inside $\hol(C)$ and assuming that
\begin{align*}
    \Phi(x)=\left(\begin{pmatrix}
    b_1&a_1\\0&1
\end{pmatrix},\begin{pmatrix}
    d_1&c_1\\0&1
\end{pmatrix}\right), 
\Phi(y)=\left(\begin{pmatrix}
    b_2&a_2\\0&1
\end{pmatrix},\begin{pmatrix}
    d_2&c_2\\0&1
\end{pmatrix}\right).
\end{align*}
From $\Phi(x)^{p^n}=e_{\hol(C)}$ we get the equations
\begin{align}
b_1^{p^n}&\equiv 1\pmod{p^n},\label{gc1}\\
a_1\left(1+b_1+\cdots+b_1^{p^n-1}\right)&\equiv 0\pmod{p^n}\label{gc2},\\
d_1^{p^n}&\equiv 1\pmod{q},\label{gc3}\\
c_1\left(1+d_1+\cdots+d_1^{p^n-1}\right)&\equiv 0\pmod{q}\label{gc4}.
\end{align}
Note that $d_1^{q-1}\equiv 1\pmod{q}$ and $(q-1,p^n)=1$. Combining this with \cref{gc3}, we get that $d_1=1$.
Then plugging $d_1=1$ in \cref{gc4}, conclude that $c_1=0$. For ensuring regularity,
we need to take $a_1$ is a unit in $\Z_{p^n}$.
Using the equation $\Phi(y)^q=1$ we get the equations
\begin{align}
    b_2^{q}&\equiv1\pmod{p^n},\label{gcy1}\\
    a_2\left(1+b_2+\cdots+b_2^{q-1}\right)&\equiv 0\pmod{p^n}\label{gcy2},\\
    d_2^{q}&\equiv 1\pmod{q},\label{gcy3}\\
    c_2\left(1+d_2+\cdots+d_2^{q-1}\right)&\equiv 0\pmod{q} \label{gcy4}.
\end{align}
Since the order of $d_2$ divides $q-1$, we get $d_2=1$ from \cref{gcy3}. 
Finally comparing both sides of the equation $\Phi(x)^k\Phi(y)=\Phi(y)\Phi(x)$ we get (using the conclusions of the preceding discussions)
\begin{align}
    b_1^{k-1}\equiv 1&\pmod{p^n}\label{gcxy1}\\
    b_2a_1+a_2\equiv b_1^ka_2+a_1\left(1+b_1+\cdots+b_1^{k-1}\right) &\pmod{p^n}\label{gcxy2}.
\end{align}
Using \cref{prop:unit}, \cref{gc1} and \cref{gcxy1} we conclude that $b_1=1$. Putting the
value of $b_1$ in \cref{gcxy2} we get that $b_2=k$. Further to ensure regularity we need to impose $c_2\neq 0$ (using a similar argument in the discussion after \cref{disc:regularity}). 
Thus the total number of regular embeddings in this case is given by $\varphi(p^n)p^n(q-1)$. 
\begin{proposition}\label{thm:G-Cyclic}
    Let $\C$ be the cyclic group of order $p^nq$ and $\G$ be the nonabelian group isomorphic to $\Z_{p^n}\rtimes\Z_q$, where $p$ and $q$ are primes. Then $e(\G,\C)=p^n$ and $e'(\G,\C)=q-1$.
\end{proposition}
\subsection{Embedding of $\C=\Z_{p^n}\times \Z_q$ in the $\hol(\Z_{p^n}\rtimes \Z_q)$}\label{subsec:eCG} Recall the description of $\hol(G)$ from \cref{subsec:eGG} and the presentation for $C$ from \cref{subsec:eGC}. Consider a homomorphism 
$\Phi:\C\longrightarrow\hol(G)$ determined by
\begin{align*}
    \Phi(r)=\left(x^{i_1}y^{j_1},\begin{pmatrix}
        \beta_1&\alpha_1\\
        0 & 1
    \end{pmatrix}\right), \Phi(s)=\left(x^{i_2}y^{j_2},\begin{pmatrix}
        \beta_2&\alpha_2\\
        0 & 1
    \end{pmatrix}\right). 
\end{align*}
Given that $\Phi(r)$ has to be an element of order $p^n$ and the embedding is regular, using a similar argument as in \cref{subsec:eGG} we 
conclude that $j_1=0$, $i_1$ is a unit in $\Z_{p_n}$ and, $j_2$ is a unit in
$\Z_q$. From $\Phi(r)^{p^n}=1$, we get that
\begin{align*}
    i_1\left(1+\beta_1+\cdots+\beta^{p^n-1}\right)&\equiv 0\pmod{p^n},\\
    \alpha_1\left(1+\beta_1+\cdots+\beta^{p^n-1}\right)&\equiv 0\pmod{p^n},\\
    \beta_1^{p^n}&\equiv 1\pmod{p^n}.
\end{align*}
From the last equation above and \textcolor{red}{\cite[Corollary 2.2]{ArPa22}} we get that $\beta_1\equiv 1\pmod{p}$. Hence the first two equations will
always be satisfied irrespective of choices of $i_1$ and $\alpha_1$. From the equation 
$\Phi(s)^q=1$, we get 
\begin{align}
     \beta_2^q &\equiv 1 \pmod{p^n},\label{ecg1}\\
 \alpha_2(1+\beta_2+\beta_2^2+\ldots +\beta_2^{q-1})&\equiv 0 \pmod{p^n}&,\label{ecg2}\\
 \ell_q &\equiv 0 \pmod{p^n},\label{ecg3}
\end{align}
where $\ell_q$ is as defined in \cref{subsec:eGG}.
Furthermore $\Phi(r)\Phi(s)=\Phi(s)\Phi(r)$ gives that
\begin{align}
    \alpha_2(\beta_1-1)&\equiv \alpha_1(\beta_2-1),\pmod{p^n}\label{ecg4}
\\
i_1+\beta_1i_2+\alpha_1k^{-1}\left(1-k^{j_2}\right)
&\equiv i_2+k^{j_2}\beta_2i_1\pmod{p^n}\label{ecg5}.
\end{align}
Let $\beta_2=k^a$ for some $a\geq 0$. We divide this into two cases $a=0$ and $a\neq 0$. 
\subsubsection{{Case I:} a=0} In this case we get $\alpha_2=0$ from \cref{ecg2}. Hence \cref{ecg4} is always satisfied. Note that 
\cref{ecg3} holds true, since $j_2+a\neq q$ by using similar arguments as of \cref{subsec:eGG}.
Putting $\beta_2=1$ in \cref{ecg5} we get $\left(i_1+\alpha_1k^{-1}\right)\left(1-k^{j_2}\right)\equiv i_2\left(1-\beta_2\right)\pmod{p^n}$. Hence the choice of $\alpha_1$ gets determined 
by those of $i_1$, $i_2$, $\beta_1$ and, $j_2$. Hence the total number 
of embedding in  this case becomes $\varphi(p^n)p^{2n-1}(q-1)$.
\subsubsection{{Case II:} $a\neq 0$} From \cref{ecg4}, substituting
$\alpha_1= {\alpha_2(\beta_1-1)}{(k^a-1)^{-1}}$ in \cref{ecg5} we get
\begin{align}
    i_1\left(k^a-1\right)\left(1-k^{j_2+a}\right)\equiv \left(1-\beta_1\right)\left(i_2\left(k^a-1\right)+\alpha_2k^{-1}\left(1-k^{j_2}\right)\right)\pmod{p^n}.\label{ecg6}
\end{align}
We claim that $j_2+a=q$. 
Indeed, if $j_2+a\neq q$,
we have that the LHS of \cref{ecg6} is a unit in $\Z_{p^n}$, whereas $(1-\beta_1)$ is never a unit (since $\beta_1\equiv 1\pmod{p^n}$). 
Next, putting $j_2+a=q$ in \cref{ecg6}, the LHS becomes $0$. Substituting $j_2+a=q$
in \cref{ecg3} we get $i_2\equiv-\alpha_2k^{-1}(1-k^{j_2})(k^{j_2}q^{-1})(1+(1+k^a)k^{j_2}+\cdots$$+(1+k^a+\cdots+k^{(q-2)a})k^{(q-2)j_2})$ $\pmod{p^n}$. Further
substituting this value of $i_2$ to \cref{ecg6}, we get that both sides of the equation become zero. 
Hence we get that in this case, the total number of regular embedding of $\C$ in $\hol(G)$ is given by $\varphi(p^n)p^{2n-1}(q-1)$.
\begin{proposition}\label{thm:G-Cyclic-number}
    Let $\C$ be the cyclic group of order $p^nq$ and $\G$ be the nonabelian group isomorphic to $\Z_{p^n}\rtimes\Z_q$. Then $e(\C,\G)=2p^{n-1}(q-1)$ and \textcolor{black}{$e'(\C,\G)=p^{2n-1}$}.
\end{proposition}
 Now \cref{theorem-p->-q} follows from \cref{thm:total-isomorphic}, \cref{thm:G-Cyclic}, and \cref{thm:G-Cyclic-number}.
\section{The case $p < q$}\label{sec:main-results-2}
In this section, we prove \cref{theorem-p-<-q}. We use methods, described in the beginning of \cref{sec:main-results-1}. In this case, there are exactly $b+1$ types of groups up to isomorphism, where $b=\mathrm{min}\{a,n\}$ with $p^a||q-1$. This section will be divided into four subsections, depending on the isomorphism types of $G=G_{b_1}$ and $N=G_{b_2}$, where $0\leq b_1,b_2\leq n$.
\subsection{Isomorphic type}\label{subsub:isom}
First, we consider the isomorphic case.
Let $G= \Z_q \rtimes_k \Z_p^n  $, where $k$ is an element of order $p^b$. We are looking at $e(G, G)$.
\subsubsection{The case $b=0$} In this case, the groups are cyclic and $e'(G,G)$, $e(G,G)$ have been enumerated in \cite[Theorem 2]{By13}.
\subsubsection{The case when $0<b \leq n$} 
Let us take a group homomorphism $\Phi:G_b\longrightarrow\hol(G_b)$ defined by
\begin{align*}
    \Phi(x)=\left(y^{j_1}x^{i_1},\left(\gamma_1,\begin{pmatrix}
        \beta_1 & \alpha_1\\
        0 & 1
    \end{pmatrix}\right)\right),
    \text{ and }\Phi(y)=\left(y^{j_2}x^{i_2},\left(\gamma_2,\begin{pmatrix}
        \beta_2 & \alpha_2\\
        0 & 1
    \end{pmatrix}\right)\right).
\end{align*}
From $\Phi(y)^q=1$ and from $\Phi(xy)= \Phi(y^kx)$, we get the relations $
   i_2=0$, $
   \beta_2=1$, $
   \gamma_2=1$ and
\begin{align}
   \alpha_2(k-\beta_1)\equiv0&\pmod{q},\label{iso1}\\
   j_2(k^{i_1-1}\beta_1-1)\equiv \alpha_2(1+k+k^2\cdots k^{i_1-1})&\pmod{q}.\label{iso2}
\end{align}

Thus if $\alpha_2=0$, then $\beta_1= k^{1-i_1}$. If $\alpha_2\neq 0$, then $\beta_1=k$ and $\alpha_2= j_2(k-1)$.
From $\Phi(x)^{p^n}=1$, we get the following equivalences in $\Z_q$.
\begin{align}
    {\beta_1}^{p^n}= 1\label{xpower1}\\
    \alpha_1(1+\beta_1+{\beta_1}^2 \cdots {\beta_1^{p^n-1}})= 0\label{xpower2}.
\end{align}
By explicit calculations, we can show that, the exponent of $y$ in $\Phi(x)^{p^n}$
is given by
\begin{align*}
\textcolor{blue}{    \mathrm{Exp}_y\left(\Phi(x)^{p^n}\right)=j_1\left(\sum\limits_{u=0}^{p^n-1}m^u\right)+\dfrac{\alpha_1}{m(k^{\gamma_1}-1)}\left\{\sum\limits_{v=1}^{p^n-1}m^{p^n-v}\left(k^{i_1\left(1+\gamma_1+\ldots+\gamma_1^{v-1}\right)}-k^{i_1}\right)\right\},}
\end{align*}
where $m=\beta_1 k^{i_1}$. Using \cref{iso1} and \cref{iso2}, we can show that $m\in\{k,k^{i_1+1}\}$
First, let us take $m=k$. Then $\sum\limits_{u=0}^{p^n-1}m^{u}\equiv 0\pmod{q}$. 
We aim to show that the other summand is also zero in $\Z_q$. We have in $\Z_q$
\begin{align*}
    \sum\limits_{v=1}^{p^n-1}m^{p^n-v}\left(k^{i_1\left(1+\gamma_1+\ldots+\gamma_1^{v-1}\right)}-k^{i}\right)
    =\sum\limits_{v=1}^{p^n}k^{i_1\left(1+\gamma_1+\ldots+\gamma_1^{v-1}\right)-v}.
\end{align*}
Note that here $i_1$ and $\gamma_1$ are fixed. Denote by $\Gamma(v)=i_1(1+\gamma_1+\ldots+\gamma_1^{v-1})-v\pmod{p^n}$.
Suppose for $1\leq v_1\neq v_2\leq p^n$ we have $\Gamma(v_1)\equiv\Gamma(v_2)\pmod{p^n}$. Then we have
$i(\gamma_1^{v_1}-\gamma_1^{v_2})\equiv(v_1-v_2)(\gamma_1-1)\pmod{p^n}$. Since the Sylow-$p$-subgroup of $\Z_{p^n}^\times$
is generated by $(1+p)$ and $\gamma_1$ is an element having $p$-power order, 
say an element of order $p^g$. Then $p^{n-g}||\gamma_1-1$. Thus $v_1-v_2\equiv0\pmod{p^g}$, using \cref{lem-power-of-p}. 
Conversely if $v_1-v_2\equiv 0\pmod{\mathrm{ord}{\gamma_1}}$, then $i(\gamma_1^{v_1}-\gamma_1{v_2})\equiv (v_1-v_2)(\gamma_1-1)\pmod{p^n}$.
Thus $\Gamma$ gives rise to a function from $\Z_{p^n}$ to the subset $\{p^g,2p^g,3p^g,\ldots,p^n\}$. Thus the sum is reduced to
   $p^g\sum\limits_{t=1}^{p^{n-g}}k^{tp^g}$. If $k^{p^g}=1$, we get the sum to be zero. Otherwise this sum is $p^g\dfrac{k^{p^n}-1}{k^{p^g}-1}=0$. This finishes
the proof.

Now, take the case when $m=k^{i_1+1}$. Then again the multiplier of $j_1$ is zero in $\Z_q$. We claim that the other summand 
is also zero in the above expression. We have in this case,
\begin{align*}
    &\sum\limits_{v=1}^{p^n-1}m^{p^n-v}\left(k^{i_1\left(1+\gamma_1+\ldots+\gamma_1^{v-1}\right)}-k^{i_1}\right)
    =\sum\limits\limits_{v=1}^{p^n}k^{i_1\left(1+\gamma_1+\ldots+\gamma_1^{v-1}-v\right)-v}-\sum\limits_{v=1}^{p^n}k^{i_1(1-v)-v}\\
    =&\begin{cases}
    \sum\limits _{v=1}^{p^n}k^{i_1\left(1+\gamma_1+\ldots+\gamma_1^{v-1}\right)-\left(i_1+1\right)v} &\text{ when }i_1+1\neq 0\pmod{p}\\
    \sum\limits _{v=1}^{p^n}k^{i_1\left(1+\gamma_1+\ldots+\gamma_1^{v-1}-v\right)-v}-\sum\limits_{v=1}^{p^n}k^{i_1(1-v)-v}&\text{otherwise}.
    \end{cases}
\end{align*}
We start by considering the first subcase, i.e. $i_1+1$ being a unit in $\Z_{p^n}$. Again denote by $\Gamma(v)=i_1\left(1+\gamma_1+\ldots+\gamma_1^{v-1}-v\right)-(i_1+1)v$.
Then $\Gamma(v_1)\equiv\Gamma(v_2)\pmod{p^n}$ implies that $i_1(\gamma_1^{v_1}-\gamma_1^{v_2})\equiv (i_1+1)(\gamma_1-1)(v_1-v_2)\pmod{p^n}$.
Then proceeding as before, we get the result. 
Next, consider the second subcase. In this case, we show that both of the sums are zero. Take 
$\Gamma_1(v)=i_1\left(1+\gamma_1+\ldots+\gamma_1^{v-1}-v\right)-(i_1+1)v$ and $\Gamma_2(v)=i_1(1-v)-v$. Assume $p^h||i_1+1$,
then $\Gamma_2(v')=\Gamma_2(v'')$ iff $v'\equiv v''\pmod{p^{n-h}}$, using \cref{lem-power-of-p}. Thus
$\Gamma_2$ determines a function to the subset $\{p^{n-h},2p^{n-h},\ldots, p^n\}$ and hence the second term of the
expression above vanishes. An argument similar to the previous cases of $\Gamma(v)$, shows that the first term is
$0$ as well in $\Z_q$. Thus we have proved the following lemma.
\begin{lemma}\label{lem-simulatenous-zero}
    In $\mathrm{Exp}_y\left(\Phi(x)^{p^n}\right)$, if the coefficient of $j_1$ is zero in $\Z_q$, then so is the coefficient of 
    $\alpha_1$.
\end{lemma}
We claim that $i_1$ is a unit. Suppose $i_1$ is not a unit. We note that $\Phi(x)^{p^{n-1}}=\left(y^{J},\left(1,\begin{pmatrix}
    1 & 0\\
    0 & 1
\end{pmatrix}\right)\right)$, for some $J$. Note that if $\beta_1=0$ then $\alpha_1=0$, otherwise $1+\beta_1+
\ldots+\beta_1^{p^n-1}=0$, whence the matrix entry is justified. 
Now, if $J=0$ then this map is not regular. Otherwise when $J\neq 0$, we get $J$ is a unit in $\Z_q$.
Since $p$ is a unit is $\Z_q$, we get that $\Phi(x_1)^{p^n}$ is not identity 
element. {This proves claim}.
Now we are ready to count the number of Hopf-Galois structures on extensions, whose group is of the form $\G_b$
for some $0<b<n$. This will be divided into four cases. Before proceeding, we note that
none of the cases, impose any condition on $j_2$ and $\gamma_1$.

\noindent\textit{Case $1$: The case $\beta_1=1$.} This implies $\alpha_2=0$. 
Since if $\alpha_2\neq 0$, then $\beta_1= k\neq1$. From $\alpha_2= 0$ we get that 
$i_1\equiv 1 \pmod{p^b}$, from which we get that $i_1$ has $p^{n-b}$ 
possibilities. Further $\alpha_1=0$ from \cref{xpower2}. In this case,  $j_1$ has $q$ possibilities since $m\neq 1$, using \cref{lem-simulatenous-zero}. Thus in this case we get $\varphi(q)q p^{2(n-b)}$  embedding.

\noindent\textit{Case $2$: The case $\beta_1\neq1$, and $\alpha_2=0$.} Note that $\alpha_2=0$ implies that $\beta_1=k^{1-i_1}$.
Also, $\beta_1\neq 1$ imposes the condition that $i_1$ has $\varphi(p^n)-p^{n-b}$ possibilities. In this case, $j_1$ and $\alpha_1$ have $q$ possibilities each. 
Thus in this case we have $\varphi(q)\left(\varphi(p^n)-p^{n-b}\right) q^2 p^{n-b}$ embeddings.

\noindent\textit{Case $3$: The case $\beta_1\neq1$, $\alpha_2\neq 0$, and $1+i_1\equiv 0\pmod{p^b}$.} Since $1+i_1\equiv 0\pmod{p^b}$, we get $m=1$. Hence the value of $j_1$ gets fixed. Thus in this case, we have $\varphi(q)q p^{2(n-b)}$  embeddings.

\noindent\textit{Case $4$: The case $\beta_1\neq1$, $\alpha_2\neq 0$, and $1+i_1\not\equiv 0\pmod{p^b}$.} In this case $i_1$ has $\varphi(p^n)-p^{n-b}$ possibilities.
Similar to Case $2$, $j_1$ has $q$ possible values. Thus in this case, we have
$\varphi(q)\left(\varphi(p^n)-p^{n-b}\right) q^2 p^{n-b}$ embeddings.

In all of the cases above, the embeddings are regular, which is guaranteed by
the conditions that $i_1$ and $j_2$ are units. Furthermore, 
In conclusion, we have proved the following result.
\begin{proposition}\label{thm:isomorphic-2}
    Let $\G_b=\Z_q\rtimes_k \Z_{p^n}$, where $k\in\Z_q$ is of order $p^b$ for some
    $0<b\leq n$. Then $e'\left(\G_b,\G_b\right)=e\left(\G_b,\G_b\right)=2\left(p^{n-b}+q\left(\varphi(p^n)-p^{n-b}\right)\right)$. 
\end{proposition}

\subsection{Non-isomorphic type} This case will be divided into three cases, depending on the values of $b_1$ and $b_2$.
\subsubsection{The case $1\leq b_1\neq b_\leq n$} We will need a variation of 
\cref{lem-simulatenous-zero}, for dealing with this case. We start with a presentation of these two groups. For $t=1$ and $2$, let us fix
\begin{align*}
    \G_{b_t}=\left\langle x_t,y_t\middle\vert x_t^{p^n}=y_t^{q}=1, x_ty_tx_t^{-1}=y_t^{k_t}\right\rangle,
\end{align*}
where $k_t$ is an element of order $p^{b_t}$.
Now we consider $\Phi:\G_{b_1}\longrightarrow
\hol\left(\G_{b_2}\right)$ is an regular embedding and $\Phi(x_1)=\left(y_2^{j_1}x_2^{i_1},\left(\gamma_1,
\begin{pmatrix}
    \beta_1 & \alpha_1\\
    0 & 1
\end{pmatrix}\right)\right)$, then it can be proved that,
\begin{align*}
\textcolor{blue}{    \mathrm{Exp}_y\left(\Phi(x)^{p^n}\right)=j\left(\sum\limits_{u=0}^{p^n-1}m^u\right)+\dfrac{\alpha_1}{m(k_2^{\gamma_1}-1)}\left\{\sum\limits_{v=1}^{p^n-1}m^{p^n-v}\left(k_2^{i_2\left(1+\gamma_1+\ldots+\gamma_1^{v-1}\right)}-k_2^{i_2}\right)\right\},}
\end{align*}
where $m=\beta_1 k_2^{i_1}$. It can be shown that $m\in\left\{k_1,k_1k_2^{i_1} \right\}$ modulo $q$, using 
\cref{iso1} and \cref{iso2}. Note that in any of the cases $b_1<b_2$ or $b_2<b_1$, $m$ is purely a power of
$k_1$ or $k_2$, since $\Z_{p^n}^\times$ is cyclic. Then a variation of the argument before \cref{lem-simulatenous-zero},
proves the following result. 
\begin{lemma}\label{lem-simul-non-iso}
In $\mathrm{Exp}_y\left(\Phi(x_1)^{p^n}\right)$ if the coefficient of $j_1$ is $0$ in $\Z_q$, then so is the coefficients of $\alpha_1$.
\end{lemma} Hoping that the reader is now familiar with the flow of arguments, without loss of generality in this case we will assume that the embedding is given by,
\begin{align*}
    \Phi(x_1) = \left(y_2^{j_1}x_2^{i_1}\left(\gamma_1,\begin{pmatrix}
        \beta_1 & \alpha_1 \\ 0 & 1
    \end{pmatrix}\right)\right), \Phi(y_1) = \left(y_2^{j_2}\left(1,\begin{pmatrix}
        1 & \alpha_2 \\ 0 & 1
    \end{pmatrix}\right)\right),
\end{align*}
where $i_1$ is a unit in $\Z_{p^n}$ (using same argument as in \cref{subsub:isom}), $\gamma_1$ is a unit in $\Z_{p^n}$ satisfying $\gamma_1^{p^{n-b_2}}=1$, and $j_2$ is a unit in $\Z_q$. 
Comparing the both sides of the equation $\Phi(x_1)\Phi(y_1)=\Phi(y_1)^{k_1}\Phi(x_1)$, we get
\begin{align}
    \alpha_2(k_1-\beta_1)\equiv 0&\pmod{q},\label{noniso1}\\
    k_2^{i_1}\beta_1j_2\equiv j_2k_1+j_2\left(1+k_2+\ldots+k_2^{i_1-1}\right)&\pmod{q}.\label{noniso2}
\end{align}
From \cref{noniso1} either $\alpha_2=0$ or $\beta_1=k_1$. Irrespective of the cases $\beta_1k_1^{i_1}\neq 1$. 
Thus from \cref{lem-simul-non-iso} $j_1$ can take any value from $\Z_{q}$.
Now, in the first case, $\beta_1=k_1k_2^{-i_1}$ (from \cref{noniso2}). 
Also $\gamma_1$ and $\alpha_1$ have $p^{n-b_2}$ and $q$ many choices respectively. This gives that
total number of embeddings in this case is given by $\varphi(q)\varphi(p^n)q^2p^{n-b_2}$. 
In the second case, $\alpha_2=(k_2-1)j_2$ and $\gamma_1$, $\alpha_1$ have $p^{n-b_2}$, $q$ many choices respectively. Thus the
total number of embeddings arising from this case is given by $\varphi(q)\varphi(p^n)q^2p^{n-b_2}$. Given that
$i_1$ and $j_2$ are units, we get that the constructed map is regular. We now have the following result.
\begin{proposition}\label{thm:non-iso-noncyclic}
    Let $\G_{b_t}=
    \Z_q\rtimes _{k_t}\Z_{p^n}$, where $k_t$ is an element of
    $\Z_{p^n}$ of order $p^{b_t}$, for $t=1$, $2$. Let $0<b_1\neq b_2\leq n$. Then 
    \begin{align*}
        e'\left(\G_{b_1},\G_{b_2}\right)= 2qp^{n+b_1-b_2-1}(p-1),~e\left(\G_{b_1},\G_{b_2}\right)=2qp^{n-1}(p-1).
    \end{align*}
\end{proposition}
\subsubsection{The case $0= b_1<b_2\leq n$} In this case $\G_{b_1}$ is cyclic and hence the 
presentations of the groups $\G_{b_1}$ and $\G_{b_2}$ are chosen to be
\begin{align*}
    \G_{b_1}=\left\langle x_1,y_1\middle\vert x_1^{p^n}=y_1^{q}=1, x_1y_1x_1^{-1}=y_1\right\rangle,
    \G_{b_2}=\left\langle x_2,y_2\middle\vert x_2^{p^n}=y_2^{q}=1, x_2y_2x_2^{-1}=y_2^{k_2}\right\rangle,
\end{align*}
with $k_2\in\Z_{p^n}$ being an element of multiplicative order $p^{b_2}$. Fix a homomorphism $\Phi:\G_{b_1}\longrightarrow \hol(\G_{b_2})$ given by
\begin{align*}
    \Phi(x_1)=\left(y_2^{j_1}x_2^{i_1}\left(\gamma_1,\begin{pmatrix}
        \beta_1 & \alpha_1 \\ 0 & 1
    \end{pmatrix}\right)\right),
    \Phi(y_1)=\left(y_2^{j_2}x_2^{i_2}\left(\gamma_2,\begin{pmatrix}
        \beta_2 & \alpha_2 \\ 0 & 1
    \end{pmatrix}\right)\right).
\end{align*}
From the condition $\Phi(y_1)^q$, we get the conditions that $i_2=0$, $\gamma_2=0$ and $\beta_2=1$. To ensure the regularity of the maps,
we will need $i_1$ and $j_2$ to be units in $\Z_{p^n}$ and $\Z_{q}$ respectively (see \cref{subsub:isom}).
Equating the two sides of the equality $\Phi(x_1)\Phi(y_1)=\Phi(y_1)\Phi(x_1)$, we get that
\begin{align}
    \alpha_2(1-\beta_1)\equiv 0 & \pmod{q},\label{nonisocyc1}\\
    \beta_1k_2^{i_1}j_2\equiv j_2+\alpha_2\left(1+k_2+\ldots+k_2^{i_1-1}\right)&\pmod{q}.\label{nonisocyc2}
\end{align}
Hence from \cref{nonisocyc1} we have either $\alpha_2=0$ or $\beta_1=1$. In case $\alpha_2=0$, plugging the value in \cref{nonisocyc2} 
we get that $\beta_1k_2^{i_1}=1$, whence $j_1$ has fixed choice, once $\alpha_1$ is fixed. 
Furthermore $\alpha_1$, $\gamma_1$ have $q$, $p^{n-b_2}$ choices.
In the case $\beta+1=1$, from \cref{nonisocyc2} we get that $\alpha_2=j_2(k_2-1)$ and $\beta_1 k_2^{i_1}\neq 1$. Hence \cref{lem-simul-non-iso} applies.
Thus $j_1$, and $\gamma_1$ have $q$, and $p^{n-b_2}$ possibilities. We conclude that in both cases the number of 
regular embedding of the cyclic group of order $p^nq$ in $\hol\left(\G_{b_2}\right)$ is given by $q\varphi(q)p^{n-b_2}\varphi(p^n)$. We have the following result.
\begin{proposition}\label{thm:cyc-nontrivial}
    Let $\C$ denotes the cyclic group of order $p^nq$ and $\G_b\cong\Z_q\rtimes_{k_b}\Z_{p^n}$, where $k_b\in\Z_q$ is an element of multiplicative order $p^{b}$. Then
    \begin{align*}
        e'(\C,\G_b)=2p^{n-b}q, \text{ and }e(\C,\G_b)=2(p-1) p^{n-1}
    \end{align*}
\end{proposition}
\subsubsection{The case $0=b_2< b_1\leq n$} Here we count the number $e'(\G_{b_1},\G_{b_2})$ (equivalently $e(\G_{b_1},\G_{b_2})$). Here $\G_{b_2}$ is a cyclic group of order $p^nq$. In this case, we have,
\begin{align*}
    \hol(\G_{b_2})\cong \left\{\left(y_2^jx_2^{i},(\omega,\delta)\middle\vert \substack{ (j,i)\in\Z_q\times\Z_{p^n}\\
    (\omega,\delta)\in\Z_q^{\times}\times \Z_{p^n}^\times}\right)\right\}.
\end{align*}
We fix an embedding $\Phi:\G_{b_1}\longrightarrow \hol\left(\G_{b_2}\right)$ determined by
\begin{align*}
    \Phi(x_1)=\left(y_2^{j_1}x_2^{i_1},\left(\omega_1,\delta_1\right)\right),~
    \Phi(x_1)=\left(y_2^{j_2}x_2^{i_2},\left(\omega_2,\delta_2\right)\right).
\end{align*}
From $\Phi(y_1)^q=1$, we get $\omega_2=1,\delta_2=1$ and $i_2=0$. Considering $\Phi(x_1)^{p^n}=1$ we get that
$\omega_1^{p^n}=1$, $\delta_1^{p^n}=1$, and
\begin{align}
    j_1\left(1+\omega_1+\ldots+\omega_1^{p^n-1}\right)\equiv 0 &\pmod{q},\label{intocyc1}\\
    i_1\left(1+\delta_1+\ldots+\delta_1^{p^n-1}\right)\equiv 0 &\pmod{p^n}.\label{intocyc2}.
\end{align}
Finally comparing both sides of the equation $\Phi(x_1)\Phi(y_1)=\Phi(y_1)^{k_1}\Phi(x_1)$, we get that $\omega_1=k_1$ and
hence \cref{intocyc1} gets satisfied automatically.
To ensure that the embedding is regular, we will need that $i_1$ and $j_2$ are units. Any choice of $\delta _1$
satisfies \cref{intocyc2}. Thus $j_1$, $j_2$, $i_1$, and $\delta_1$
have $\varphi(q)$, $q$, $\varphi(p^n)$, and $p^{n-1}$ possibilities respectively. We conclude with the following result.
\begin{proposition}\label{thm:into-cyc}
    Let $\G_b\cong\Z_q\rtimes_{k_b}\Z_{p^n}$, where $k_b$ is an element of $\Z_q$ of order $p^b$, $1\leq b\leq n$, and $\C$ denote the cyclic group of order $p^nq$. Then we have
    \begin{align*}
    e'\left(\G_b,\C\right)=p^{n+b-2}(p-1),~
    e\left(\G_b,\C\right)=p^{n-1}q.
    \end{align*}
\end{proposition}
The \cref{theorem-p-<-q} follows from \cref{thm:isomorphic-2}, \cref{thm:non-iso-noncyclic}, \cref{thm:cyc-nontrivial}, and \cref{thm:into-cyc}.

\section{Realizable pair of groups}\label{sec:conclusion}
Given two finite groups $G$ and $N$ of the same order, we say that the pair $(G,N)$ is \emph{realizable} if there exists a Hopf-Galois structure on a Galois $G$-extension, of type $N$. In other words a pair $(G,N)$ is realizable if $e(G,N) \neq 0$. This is equivalent to saying there exists a skew brace with its multiplicative group isomorphic to $G$ and its additive group isomorphic to $N$. This problem is not well studied since given an integer $n$, the classification of all the groups of
size $n$ is not known. However, they have been studied for a variety of groups. When $G$ is a cyclic group of odd order and the pair $(G,N)$ is realizable then in \cite{By13}, the author showed that if $N$ is abelian then it is cyclic. If $N$ is a non ableian simple group and $G$ is a solvable group with the pair $(G,N)$ being realizable, then in \cite{Ts23a} $N$ was completely classified. Whenever $N$ or $G$ is isomorphic to $\Z_n\rtimes \Z_2$ for an odd $n$ then their realizabilities were studied in \cite{ArPa23}.  

Among a few available techniques, the notion of {bijective crossed homomorphism} to study realizability problems for a pair of groups of the same order,
was introduced by Tsang in the work \cite{Ts19}.
Given an element $\fa\in\ha(G,\aut(N))$, a map $\ga\in \ma(G,N)$ is said to be a \emph{crossed homomorphism} with respect to $\fa$ if 
$    \ga(ab)=\ga(a)\fa(a)(\ga(b))\text{ for all }a,b\in G$.
Setting
$Z_{\fa}^1(G,N)=\{\ga:\ga$ is bijective crossed homomorphism with respect to $\fa\}$, we have the following two results.
\begin{proposition}\cite[Proposition 2.1]{Ts19}\label{p002}
The regular subgroups of $\hol(N)$ which are isomorphic to $G$ are precisely the subsets of $\hol(N)$ of the form
$\{(\ga(a),\fa(a)):a\in G\},$
where $\fa\in\ha(G,\aut(N)),\ga\in Z_{\fa}^1(G,N)$.
\end{proposition}
\begin{proposition}\label{t002}\cite[Proposition 3.3]{TsQi20}
Let $G,N$ be two groups such that $|G|=|N|$. Let $\mathfrak{f}\in\ha(G,\aut(N))$ and $\mathfrak{g}\in Z_\mathfrak{f}^1(G,N)$ be a bijective crossed homomorphism (i.e. $(G,N)$ is realizable). Then if $M$ is a characteristic subgroup of $N$ and $H=\mathfrak{g}^{-1}(M)$, we have
that the pair $(H,M)$ is realizable.
\end{proposition}
We will need the following two results, where the realizability of cyclic groups have been characterized.
We will use modifications of these characterizations towards proving the realizability of groups of the form $\Z_{p^n}\rtimes\Z_{q}$.
\begin{proposition}\cite[Theorem 3.1]{Ts22}\label{p003}
Let $N$ be a group of odd order $n$ such that the pair $(\Z_n,N)$ is realizable. Then $N$ is a $C$-group (i.e. all the Sylow subgroups are cyclic).
\end{proposition}
\begin{proposition}\cite[Theorem 1]{Ru19}\label{p004}
Let $G$ be a group of order $n$ such that $(G,\Z_n)$ is realizable. Then $G$ is solvable and almost Sylow-cyclic (i.e. its Sylow subgroups of odd order are cyclic, and every Sylow-$2$ subgroup of G
has a cyclic subgroup of index at most $2$).
\end{proposition}

\begin{theorem}
 Let $\N$ be a group of order $qp^n$, where $q$ is a prime, $q< p$ and  $(q,p)=1$. Then the pair $(\Z_{p^n}\rtimes\Z_q,\N)$ (or $(N,\Z_{p^n}\rtimes \Z_q)$) is realizable if and only if $\N \cong \Z_{p^n}\rtimes \Z_q$.  
\end{theorem}
\begin{proof}
Let $(\Z_{p^n}\rtimes\Z_q,\N)$ be realizable. By \cref{p002} there exists a bijective 
    crossed homomorphism $\ga\in Z^1_{\fa}(\Z_{p^n}\rtimes\Z_q,N)$ for some $\fa\in \ha(\Z_{p^n}\rtimes\Z_q,\aut(N))$.
    Let $H_p$ be the Sylow-$p$ subgroup of $N$ (it is unique since $q < p$). Then using \cref{t002} the pair $(\ga^{-1}H_{p},H_{p})$ is realizable. 
 Note that $\Z_{p^n}\rtimes\Z_q$ has unique subgroup of order $p^n$, which is cyclic. 
 This implies that $(\Z_{p^n},H_p)$ is realizable. 
Hence by \cref{p003} we get that $H_p$ is isomorphic to $\Z_{p^n}$ and therefore $N \cong \Z_{p^n}\rtimes \Z_q $. Conversely if $N \cong \Z_{p^n}\rtimes \Z_q $ then the pair $(\Z_{p^n}\rtimes\Z_q,\N)$ is realizable since $e(\Z_{p^n}\rtimes\Z_q,\N)$ is non-zero from \cref{sec:main-results-1}.

Now if the pair $(N,\Z_{p^n}\rtimes\Z_q)$ is realizable, by \cref{p002} there exists a bijective 
crossed homomorphism $\ga\in Z^1_{\fa}(N,\Z_{p^n}\rtimes\Z_q)$ for some $\fa\in \ha(G,\aut(\Z_n\rtimes\Z_2))$. 
Since $\Z_{p^n}$ is a characteristic subgroups of $\Z_{p^n}\rtimes\Z_{q}$, we get that $\ga^{-1}(\Z_{p^n})$ is a subgroup 
of $N$ and $(\ga^{-1}(\Z_{p^n}),\Z_{p^n})$ is realizable. Then by \cref{p004}, we have that 
$\ga^{-1}(\Z_{p^n})$ is almost Sylow-cylic and therefore isomorphic to $\Z_{p^n}$. Hence $N \cong \Z_{p^n}\rtimes \Z_q $. Conversely if $N \cong \Z_{p^n}\rtimes \Z_q $, then by \cref{sec:main-results-1} we have the pair $(N,\Z_{p^n}\rtimes\Z_q)$ is realizable. 
\end{proof}

\printbibliography
\end{document}